\documentclass[a4paper,12pt]{amsart}
\AtBeginDocument{%
   \def\MR#1{}
}
\usepackage{amsmath,amssymb,amsthm}
\usepackage{enumitem}
\setlist[enumerate]{font=\rm}
\usepackage{color}
\usepackage[all]{xy}
\usepackage{extpfeil}
\usepackage{tikz}
\usetikzlibrary{arrows}
\usepackage{geometry}
\newtheorem{theorem}{Theorem}[section]
\newtheorem{lemma}[theorem]{Lemma}
\newtheorem{proposition}[theorem]{Proposition}
\newtheorem{corollary}[theorem]{Corollary}

\theoremstyle{definition}
\newtheorem{definition}[theorem]{Definition}
\newtheorem{definitiontheorem}[theorem]{Definition--Theorem}
\newtheorem{example}[theorem]{Example}
\newtheorem{remark}[theorem]{Remark}

\newcommand{\lmod}{\operatorname{\mathsf{\hspace{-2pt}-mod}}}
\newcommand{\lproj}{\operatorname{\mathsf{\hspace{-2pt}-proj}}}
\newcommand{\thick}{\operatorname{\mathsf{thick}}}
\newcommand{\silt}{\operatorname{\mathsf{silt}}}
\newcommand{\twosilt}{\operatorname{\mathsf{2-silt}}}
\newcommand{\twotilt}{\operatorname{\mathsf{2-tilt}}}

\newcommand{\Rad}{{\operatorname{Rad}\nolimits}}

\newcommand{\Ext}{\operatorname{Ext}\nolimits}

\newcommand{\Hom}{\operatorname{Hom}\nolimits}
\newcommand{\induc}{{\operatorname{Ind}\nolimits}}
\newcommand{\restr}{{\operatorname{Res}\nolimits}}
\newcommand{\add}{\operatorname{\mathsf{add}}}
\newcommand{\stautilt}{\operatorname{\mathrm{s\tau-tilt}}}

\newcommand{\inertiagp}{I}
\newcommand{\decompgp}{I}

\title{
    On support $\tau$-tilting modules over blocks covering cyclic blocks
    }
\author{Ryotaro KOSHIO}
\address[Ryotaro KOSHIO]{Graduate School of Science, Department of Mathematics,  
Tokyo University of Science}
\email{1120702@ed.tus.ac.jp}
\author{Yuta KOZAKAI}
\address[Yuta KOZAKAI]{Department of Mathematics,  
Tokyo University of Science}
\email{kozakai@rs.tus.ac.jp}
\date{\today}

\begin{document}

\begin{abstract}
    Support $\tau$-tilting modules correspond to some classes of categorical objects bijectively, such as two-term tilting complexes for any finite dimensional symmetric algebra.
    This fact motivates us to classify support $\tau$-tilting modules over blocks of finite groups. Therefore we classify support $\tau$-tilting modules over particular blocks of finite groups by using the modular representation theoretical approaches including Clifford theory, Green's indecomposability theorem and so on.
\end{abstract}
\maketitle

\section{Introduction}
Adachi--Iyama--Reiten introduced the notion of $\tau$-tilting modules and support $\tau$-tilting modules, which are generalizations of the class of tilting modules in \cite{MR3187626}.
Thanks to this, we get to be able to make nontrivial considerations on the classes of algebras such as self-injective algebras whose tilting modules coincide with progenerators in Morita theory.
Let $\Lambda$ be a finite dimensional symmetric algebra (for example, a group algebra of a finite group or its block).
We denote by $\stautilt \Lambda$ the set of isomorphism classes of basic support $\tau$-tilting $\Lambda$-modules.
By many researchers, it is shown that support $\tau$-tilting modules over $\Lambda$ correspond bijectively to various classes of categorical objects associated with $\Lambda$, such as the set of two-term silting complexes in $K^b(\Lambda\lproj)$ \cite{MR3187626}, the set of functorially finite torsion classes in $\Lambda\lmod $ \cite{MR3187626}, the set of two-term simple-minded collections in $D^b(\Lambda\lmod)$ \cite{MR3178243}, the set of intermediate t-structures in $D^b(\Lambda\lmod)$ \cite{MR3220536}, the set of left finite semibricks in $\Lambda\lmod$ \cite{10.1093/imrn/rny150}, and so on. In particular, since $\Lambda$ is a finite dimensional symmetric algebra, silting complexes in $K^b(\Lambda\lproj)$ are tilting complexes by \cite{MR2927802}.
Our interest is the set of two-term tilting complexes over $\Lambda$, which is denoted by $\twotilt \Lambda$. In \cite{MR3187626}, it is shown that the set $\stautilt \Lambda$ admits  a partial order structure which is consistent with the one to one correspondences between $\stautilt \Lambda$ and $\twotilt \Lambda$.
The class of two-term tilting complexes contains the one of Okuyama--Rickard tilting complexes, which plays an important role in the study of derived equivalences of symmetric algebras \cite{okuyama1997some}.
Moreover, the information of two-term tilting complexes may allow us to classify all tilting complexes in some cases
(for example,
see \cite{MR3687098}).
For these reasons, the study of two-term tilting complexes over blocks of group algebras may provide some approaches for the solution of the famous conjecture called Brou\'e's Abelian Defect Group Conjecture.
Therefore, we believe that considerations on support $\tau$-tilting modules are also quite meaningful for modular representation theory.
At present, however, there have been few studies on the relationship between $\tau$-tilting theory and modular representation theory of finite groups, while $\tau$-tilting theory has been actively
studied by many researchers.

Let $G$ be a finite group, $k$ an algebraically closed field of characteristic $p$ and $B$ a block of $kG$.
If $B$ has a cyclic defect group (or equivalently, $B$ is of finite representation type), then $B$ is a simple algebra or a Brauer tree algebra (see \cite[Chapter 5]{MR860771}).
Since $\tau$-tilting theory for Brauer tree algebras has been well studied (for example, \cite{MR3461065}, \cite{aoki2018torsion}, \cite{ASASHIBA2020119}, \cite{aokibinomtautilt} and \cite{MR3848421}),
there are a lot of information to consider the poset structure of $\stautilt B$ in the case where $B$ has a cyclic defect group.
If $p=2$ and $B$ has a dihedral, semidihedral, or generalized quaternion defect group (or equivalently, $B$ is of tame representation type), then $B$ is a Brauer graph algebra appearing in the lists in Erdmann's classification of tame blocks (for example, see \cite{MR1064107} and \cite{MR0472984}).
The structures of partially ordered sets of support $\tau$-tilting modules over algebras appearing in Erdmann's lists were calculated in \cite{MR3856858}.
Moreover, in the same paper, they proved the following theorem.

\begin{proposition}[{See the proof of \cite[Theorem 15]{MR3856858}}]\label{Eisele Theorem}
    Let $P$ be a $p$-group and $B$ a block of $kG$. Then we get an isomorphism $\stautilt B \cong \stautilt( B\otimes_k kP)$ as partially ordered sets.
\end{proposition}
\noindent Note that $B \otimes _ k kP$ is a block of $k{\widetilde G}$ where ${\widetilde G} = G \times P $, so that Proposition \ref{Eisele Theorem} means that the calculation of $\stautilt(B \otimes _ k kP)$ can be reduced to that of $\stautilt B$.
Inspired by this result, we ask the following question:

\begin{itemize}[leftmargin=-1pt]
    \item[ ]\textit{Under what conditions are the posets of support $\tau$-tilting modules over two blocks isomorphic?}
\end{itemize}

\noindent In this paper, we will use a modular representation theoretical approach to address the question above and show that similar consequence of Proposition \ref{Eisele Theorem} holds under more general situations.

In order to state our main results, we need some notation.
Let ${\widetilde G}$ be a finite group with normal subgroup $G$, $B$ a block of $kG$ and ${\widetilde B}$ a block of $k{\widetilde G}$ covering $B$ (i.e. the condition $1_B1_{\widetilde B}\neq 0$ holds). We denote by $\inertiagp_{\widetilde G}(B)$ the inertial group of the block $B$ in ${\widetilde G}$. We say that a $kG$-module $U$ is $\inertiagp_{\widetilde G}(B)$-invariant if $xU \cong U$ as $kG$-modules for any $x\in \inertiagp_{\widetilde G}(B)$. We denote by $\induc_G^{\widetilde G}:=k{\widetilde G}\otimes_{kG}-$ the induction functor from $kG\lmod$ to $k{\widetilde G}\lmod$.
\begin{theorem}[{See Theorem \ref{Main Theorem: poset struc}}]\label{Inrtro: Main theorem poset iso}
    Assume that ${\widetilde G}/G$ is a $p$-group and that $B$ satisfies the following conditions:
    \begin{enumerate}[label=(\Roman*)
            ,font=\rm]
        \item any indecomposable $B$-module is $\inertiagp_{\widetilde G}(B)$-invariant.\label{Main Theorem: poset struc: inertia invariant}
        \item the block $B$ is $\tau$-tilting finite (i.e. $\#\stautilt B < \infty$).\label{Main Theorem: poset struc: tau-tilt finite}
    \end{enumerate}
    Then the induction functor $\induc_G^{\widetilde G}$ induces an isomorphism from $\stautilt B$ to $\stautilt {\widetilde B}$ of partially ordered sets.
\end{theorem}
If $B$ has a cyclic defect group, then the conditions \ref{Main Theorem: poset struc: inertia invariant} and \ref{Main Theorem: poset struc: tau-tilt finite} hold for $B$ automatically (see Lemma \ref{Cyclicdefect p-power index}).
Moreover, in that case the block $B$ is a Brauer tree algebra or simple algebra, thus the number of elements in $\stautilt B$ is equal to $\binom{2e}{e}$, where $e$ is the number of isomorphism classes of simple $B$-modules and $\binom{2e}{e}$ means the binomial coefficient (\cite{ASASHIBA2020119}, \cite{aokibinomtautilt}).
Combining Theorem \ref{Inrtro: Main theorem poset iso} with these facts, we get the following.
\begin{theorem}\label{Main theorem brauer tree}
    Assume that ${\widetilde G}/G$ is a $p$-group. If $B$ has a cyclic defect group, then $\stautilt B$ and $\stautilt {\widetilde B}$ are isomorphic partially ordered sets. In particular, we get $\#\stautilt {\widetilde B}=\binom{2e}{e}$ where $e$ is the number of isomorphism classes of simple $B$-modules.
\end{theorem}
Using the correspondence between support $\tau$-tilting modules and two-term tilting complexes given by \cite{MR3187626}, we get the following.

\begin{corollary}[See Theorem \ref{Induction 2tilt}]
    Assume that ${\widetilde G}/G$ is a $p$-group and that $B$ satisfies the conditions \ref{Main Theorem: poset struc: inertia invariant} and \ref{Main Theorem: poset struc: tau-tilt finite} in Theorem \ref{Inrtro: Main theorem poset iso}. Then the induction functor $\induc_G^{\widetilde G}$ induces an isomorphism $\twotilt B \cong \twotilt {\widetilde B}$ of partially ordered sets which commutes the following diagram
    \begin{equation*}
        \xymatrix{
        \stautilt B \ar[r]^-\sim_-{\induc_G^{\widetilde G}}\ar[d]_-[@]{\sim} \ar@{}[dr]|\circlearrowleft &\stautilt {\widetilde B} \ar[d]^-[@]{\sim}\\
        \twotilt B \ar[r]_-\sim^-{\induc_G^{\widetilde G}}& \twotilt {\widetilde B}
        }
    \end{equation*}
    of partially ordered sets, where the vertical maps are isomorphisms given by the above correspondence.
\end{corollary}

In this paper, we use the following notation.
Modules mean finitely generated left modules and complexes mean cochain complexes.
For  a finite dimensional algebra $\Lambda$ over a field $k$ and a $\Lambda$-module $M$, we denote by $\Rad(M)$ the Jacobson radical of $M$, by $P(M)$ the projective cover of $M$, by $\Omega(M)$ the syzygy of $M$ and $\tau M$ the Auslander--Reiten translate of $M$.
We denote by $\Lambda\lmod$ the module category of $\Lambda$ and by $K^b(\Lambda\lproj)$ the homotopy category consisting of bounded complexes of projective $\Lambda$-modules.
For an object $X$ in $\Lambda\lmod$ (or of $K^b(\Lambda\lproj)$),
we denote by $\add X$ the full subcategory of $\Lambda\lmod$ (or of $K^b(\Lambda\lproj)$, respectively) whose objects are finite direct sums of direct summands of $X$.
We say that $X$ is basic if any two indecomposable direct summands of $X$ are non-isomorphic.
For $\Lambda$-modules $U$ and $U'$, we write $U\mid U'$ if $U$ is isomorphic to a direct
summand of $U'$ as a $\Lambda$-module.

This paper organized as follows. In Section \ref{section:Preliminaries for tau tilting theory} we recall some definitions and properties of $\tau$-tilting theory. In Section \ref{Preliminaries: Modular rep-Theory} we give basic definitions and propositions of modular representation theory of finite groups. In Section \ref{section: proof of the main theorems} we prove the main theorems and its corollary, and in Section \ref{section:Applications} we give some applications and examples.

\section{Preliminaries for $\tau$-tilting theory}\label{section:Preliminaries for tau tilting theory}
In this section, let $k$ be an algebraically closed field and $\Lambda$ a finite dimensional $k$-algebra. We denote by $\tau$ the Auslander--Reiten translation. For a $\Lambda$-module $M$, we denote by $|M|$ the number of isomorphism classes of indecomposable direct summands of $M$.
\subsection{Support $\tau$-tilting module and mutation}\label{subsection: Preliminaries: Support tau tilt and mutation}
In this subsection, we recall some definitions and basic properties of support $\tau$-tilting modules.
\begin{definition}[{\cite[Definition 0.1]{MR3187626}}]
    Let $M$ be a $\Lambda$-module.
    \begin{enumerate}
        \item We say that $M$ is {\it $\tau$-rigid} if $\Hom_{\Lambda}(M,\tau M)=0$.
        \item We say that $M$ is {\it $\tau$-tilting} if $M$ is a $\tau$-rigid module and $|M|=|\Lambda|$.
        \item We say that $M$ is {\it support $\tau$-tilting} if there exists an idempotent $e$ of $\Lambda$ such that $M$ is a $\tau$-tilting $\Lambda/\Lambda e\Lambda $-module.
    \end{enumerate}
\end{definition}

We remark that any $\tau$-tilting module is a support $\tau$-tilting module because $e=0$ is an idempotent of $\Lambda$ satisfying the condition 3.

\begin{definition}[{\cite[Definition 0.3]{MR3187626}}]
    Let $M$ be a $\Lambda$-module and $P$ a projective $\Lambda$-module.
    \begin{enumerate}
        \item We say that the pair $(M,P)$ is {\it $\tau$-rigid} if $M$ is $\tau$-rigid and $\Hom_\Lambda(P,M)=0$.
        \item We say that the pair $(M,P)$ is {\it support $\tau$-tilting} (or {\it almost complete support $\tau$-tilting}) if the pair $(M,P)$ is $\tau$-rigid and $|M|+|P|=|\Lambda|$ (or $|M|+|P|=|\Lambda|-1$, respectively).
    \end{enumerate}
\end{definition}

\begin{proposition}[{\cite[Proposition 2.3]{MR3187626}}]\label{stauTilt module and pair}
    Let $(M,P)$ be a pair with a $\Lambda$-module $M$ and a projective $\Lambda$-module $P$. Let $e$ be an idempotent of $\Lambda$  such that $\add P= \add \Lambda e$.
    \begin{enumerate}
        \item The pair $(M,P)$ is a $\tau$-rigid (or support $\tau$-tilting) pair if and only if $M$ is a $\tau$-rigid (or $\tau$-tilting, respectively) $\Lambda/\Lambda e \Lambda$-module.
        \item If $(M,P)$ and $(M,Q)$ are support $\tau$-tilting pairs for some projective $\Lambda$-module $Q$, then $\add P= \add Q$.
    \end{enumerate}
\end{proposition}

\begin{remark}
    We can associate any basic support $\tau$-tilting module to a basic support $\tau$-tilting pair bijectively by Proposition \ref{stauTilt module and pair}. Hence we can consider support $\tau$-tilting pairs instead of support $\tau$-tilting modules.
\end{remark}

The following proposition gives the definitions of mutations of support $\tau$-tilting pairs and support $\tau$-tilting modules. We recall that we say a pair $(V,Q)$ of a $\Lambda$-module $V$ and a projective $\Lambda$-module $Q$ is basic if $V$ and $Q$ are basic.

\begin{proposition}[{\cite[Theorem 2.18]{MR3187626}}] \label{support tau tilting mutation theorem}
    If $(V,Q)$ is a basic almost complete support $\tau$-tilting pair, then there exist exactly two basic support $\tau$-tilting pairs containing $(V,Q)$ as a direct summand.
\end{proposition}

\begin{definition}[{\cite[Definition 2.19]{MR3187626}}]\label{Definition stautilt pair mutation}
    Let $M$ be a basic support $\tau$-tilting module, $(M,P)$ the corresponding basic support $\tau$-tilting pair and $X$ an indecomposable summand of either $M$ or $P$. For the basic almost complete support $\tau$-tilting pair $(V,Q)$ satisfying either $M \cong V\oplus X$ or $P \cong Q \oplus X$, by Proposition \ref{support tau tilting mutation theorem}, there exist a unique basic support $\tau$-tilting pair $(M',P')$ distinct to $(M,P)$ and having $(V,Q)$ as a direct summand.
    \begin{enumerate}
        \item We denote $(M',P')$ by $\mu_X(M,P)$ and call it a {\it mutation} of $(M,P)$ with respect to $X$.
        \item  We denote $M'$ by $\mu_X(M)$ and call it a  {\it mutation} of $M$ with respect to $X$.
    \end{enumerate}
\end{definition}

\subsection{Poset structures and connections with silting theory}

We denote by $\stautilt \Lambda$ the set of isomorphism classes of basic support $\tau$-tilting $\Lambda$-modules.
We can define a partial order on $\stautilt \Lambda$ as follows.
\begin{definitiontheorem}[{\cite[Theorem 2.7, 2.18, Definition-Proposition 2.28]{MR3187626}}]
    For $M, M' \in \stautilt \Lambda$, we write $M\geq M'$ if there exist a non-negative integer $r$ and an epimorphism $\varphi \colon M^{\oplus r}\rightarrow M'$. If $M$ and $M'$ are mutation of each other, then either $M > M'$ or $M < M'$ holds. Moreover, the following conditions are equivalent:
    \begin{enumerate}
        \item $M$ and $M'$ are mutation of each other, and $M>M'$.
        \item $M>M'$ and there is no support $\tau$-tilting $\Lambda$-module $L$ such that $M>L>M'$.
    \end{enumerate}
\end{definitiontheorem}

We denote $\mathcal{H}(\stautilt \Lambda)$ the Hasse quiver (Hasse diagram) for the partially ordered set $\stautilt \Lambda$. The theorem  above implies that any arrow in $\mathcal{H}(\stautilt \Lambda)$ corresponds to a support $\tau$-tilting mutation.
\begin{remark}\label{2020-03-20 15:46:59}
    The underlying graph of $\mathcal{H}(\stautilt \Lambda)$ is a $|\Lambda|$-regular graph because we can take $|\Lambda|$ sorts of mutations for each support $\tau$-tilting module.
\end{remark}
The next proposition plays an important role to prove our main theorems.

\begin{proposition}[{\cite[Corollary 2.38]{MR3187626}}]\label{connected component}
    If $\mathcal{H}(\stautilt \Lambda)$ has a finite connected component $\mathcal{C}$, then $\mathcal{H}(\stautilt \Lambda)$ coincides with $\mathcal{C}$.
\end{proposition}

Now we recall the definition of silting complexes  which is a generalization of tilting complexes. The concept of silting complexes is originated from \cite{MR976638}, and recently there have been many papers on silting complexes starting with \cite{MR2927802}.
In particular, in \cite{MR3187626}, it is shown that there is a one-to-one correspondence between two-term silting complexes and support $\tau$-tilting modules.

\begin{definition}
    Let $P$ be a complex in $K^b(\Lambda\lproj)$.
    \begin{enumerate}
        \item We say that $P$ is {\it presilting} (or {\it pretilting}) if $\Hom_{K^b(\Lambda\lproj)}(P,P[i])=0$ for any $i>0$ (or for any $i \neq 0$, respectively).
        \item We say that $P$ is {\it silting} (or {\it tilting}) if it is presilting (or pretilting, respectively) and satisfies $\thick P=K^b(\Lambda\lproj)$ where $\thick P$ is the full subcategory of $K^b(\Lambda\lproj)$ generated by $\add P$ as triangulated category.
    \end{enumerate}
\end{definition}

\begin{definition}[{\cite[Definition 2.10]{MR2927802}}]
    For a finite dimensional algebra $\Lambda$, let $\silt \Lambda$ be the set of isomorphism classes of basic silting complexes in $K^b(\Lambda \lproj)$. We can define a partial order on $\silt \Lambda$ as follows: for $P,Q \in \silt \Lambda$, we write $P\geq Q$ if
    \begin{equation*}
        \Hom_{K^b(\Lambda\lproj)}(P,Q[i])=0,
    \end{equation*}
    for any $i>0$.
\end{definition}

\begin{definition}
    We say that a complex $P \in K^b(\Lambda\lproj)$ is {\it two-term} if $P^i = 0$ for all $i \neq 0, -1$. We denote by $\twosilt \Lambda$ the subset of $\silt \Lambda$ consisting of all isomorphism classes of basic two-term silting complexes in $K^b(\Lambda \lproj)$.
\end{definition}

\begin{theorem}[{\cite[Theorem 3.2 and Corollary 3.9]{MR3187626}}]\label{2020-03-27 17:22:17}
    There is an isomorphism
    \begin{equation*}
        \stautilt \Lambda \rightarrow \twosilt \Lambda
    \end{equation*}
    of partially ordered sets given by $\stautilt \Lambda \ni (M,P) \mapsto (P_1\oplus P \xrightarrow{(f_1\ 0)} P_0) \in \twosilt \Lambda$ where $P_1\xrightarrow{f_1}P_0\xrightarrow{f_0}M\rightarrow 0$
    is a minimal projective presentation of $M$.
\end{theorem}

We remark that the correspondence above commutes with support $\tau$-tilting mutations and silting mutations \cite[Corollary 3.9]{MR3187626}.

\begin{remark}
    If $\Lambda$ is a finite dimensional symmetric $k$-algebra, then any silting complex in $K^b(\Lambda \lproj)$ is in fact a tilting complex by \cite[Example 2.8]{MR2927802}.
\end{remark}

By this fact, silting complexes over the blocks of group algebras are tilting complexes in fact. Hence the classifications of support $\tau$-tilting modules over the blocks means those of two-term tilting complexes.

\section{Preliminaries for modular representation theory}\label{Preliminaries: Modular rep-Theory}

In this section, let $G$ be a finite group and $k$ an algebraically closed field of characteristic $p$, where $p$ is a prime.
The field $k$ can always be regarded as an $kG$-module by defining $gx=x$ for
all $g\in G$ and $x \in k$. This module is called the {\it trivial module} and is denoted by $k_G$.

\subsection{Restriction functors and induction functors}
Let $H$ be a subgroup of $G$.
We denote by $\restr_{H}^G$ the {\it restriction functor} from $kG\lmod$ to $kH\lmod$ and $\induc_H^G:={}_{kG}kG\otimes _{kH} -$ the {\it induction functor} from $kH\lmod$ to $kG\lmod$.
These are exact functors and have the following properties.
The first one is called Frobenius reciprocity.
\begin{proposition}[{See \cite{MR860771}}]\label{Theorem: Frobenius and projective}
    Let $H$ be a subgroup of $G$.
    Then the functors $\restr_H^G$ and $\induc_H^G$ have the following properties:
    \begin{enumerate}
        \item the functor $\induc_H^G$ is both left and right adjoint to $\restr_H^G$.
        \item the functors $\restr_H^G$ and $\induc_H^G$ send projective modules to projective modules.
    \end{enumerate}
\end{proposition}

Let $H$ be a subgroup of $G$ and $U$ a $kH$-module. For $g \in G$, we define a $k[gHg^{-1}]$-module $gU$ consisting of symbols $gu$ where $u\in U$ as a set and its $k[gHg^{-1}]$-module structure is given by $gu+gu':=g(u+u')$, $s(gu):=g(su)$ and $ghg^{-1}(gu):=g(hu) $ for any $u, u'\in U$, $s\in k$ and $ghg^{-1} \in gHg^{-1}$. We remark that if $H$ is a normal subgroup of $G$, then $gU$ is also a $kH$-module.

\begin{remark}
    By easy calculations, we have $g\restr_H^K U \cong \restr_{gHg^{-1}}^{gKg^{-1}}gU$ where $H \leq K\leq G$.
\end{remark}

Let $H$ and $K$ be subgroups of $G$. We denote by $[G/H]$, $[K \backslash G]$ and $[K\backslash G/ H]$ some sets of representatives of $G/H$, $K \backslash G$ and $K\backslash G/ H$, respectively.

\begin{theorem}[Mackey's decomposition formula]\label{Mackey's decomposition formula}
    Let $H$ and $K$ be subgroups of $G$, and $U$ a $kH$-module. Then we have
    \begin{equation*}
        \restr^G_K\induc^G_H U\cong\bigoplus_{g \in [K\backslash G/H]}\induc^K_{K\cap gHg^{-1}}\restr^{gHg^{-1}}_{K\cap gHg^{-1}} gU.
    \end{equation*}
\end{theorem}

\begin{remark}\label{Remark: normal mackey}
    Let $N$ be a normal subgroup of $G$ and $H$ a subgroup of $G$ containing $N$. Then $N \backslash G/H=G/H$ and
    \begin{equation*}
        \restr^G_N\induc^G_H U\cong\bigoplus_{g \in [G/H]}g \restr^H_N U.
    \end{equation*}
    In particular, if $N=H$ then
    \begin{equation*}
        \restr^G_N\induc^G_N U\cong\bigoplus_{g \in [G/N]}gU.
    \end{equation*}
\end{remark}

Let $U$ and $V$ be $kG$-modules. The $k$-module $U\otimes V =U \otimes_k V$ has a $kG$-module structure given by $g(u\otimes v)=gu\otimes gv$, for all $g \in G$, $u \in U$ and $v \in V$. We remark that we have the following natural isomorphisms of $kG$-modules: $U\otimes k_G\cong U$ and $\induc_H^G(\restr_H^G U\otimes W)\cong U\otimes \induc_H^G W$ for $kG$-module $U$ and $kH$-module $W$ (for example, see \cite[Lemma 8.5]{MR860771}).

For a normal subgroup $N$ of $G$ and $kN$-module $U$, we denote by $\decompgp_G(U)$ the inertial group of $U$ in $G$, that is $\decompgp_G(U):=\left\{ g \in G \mid gU \cong U \text{ as $kN$-modules} \right\}$.
\begin{theorem}[Clifford's Theorem of simple modules]\label{Clifford's Theorem of simple module}
    Let $N$ be a normal subgroup of $G$, $S$ a simple $kG$-module and $T$ a simple $kN$-submodule of $\restr_N^G S$. Then we have a $kN$-module isomorphism
    \begin{equation*}
        \restr_N^G S \cong \bigoplus_{g \in [G/\decompgp_G(T)]} gT^{\oplus r}
    \end{equation*}
    where $r$ is an integer, called the ramification index of $S$ in $G$. Furthermore, we can consider that $T^{\oplus r}$ is a simple $k\decompgp_G(T)$-module and $S \cong \induc_{\decompgp_G(T)}^G T^{\oplus r}$ as $kG$-modules.
\end{theorem}

From now on, we will consider the case where $G/N$ is a $p$-group.
The following theorem makes substantial contribution in this paper.

\begin{theorem}[{Green's indecomposability theorem \cite{MR131454}}]\label{Green's indecomposability theorem}
    If $N$ is a normal subgroup of $G$ such that $G/N$ is a $p$-group, then $\induc_N^G V$ is an indecomposable $kG$-module for any indecomposable $kN$-module $V$.
\end{theorem}

\begin{proposition}[{See \cite[Theorem 3.5.11]{MR998775} or \cite[Exercise 19.1]{MR860771}}]\label{unique extension of simple module}
    Let $N$ be a normal subgroup of $G$ and $T$ a simple $kN$-module such that $\decompgp_G(T)=G$. If $G/N$ is a $p$-group, then there exists a unique simple $kG$-module $S$ such that $\restr_N^G S\cong T$.
\end{proposition}

\begin{lemma}\label{Lemma simple ind res}
    Let $S$ be a simple $kG$-module. Suppose that $G/N$ is a $p$-group. Then there is only $S$ that can be a composition factor of $\induc_N^G \restr_N^G S$.
\end{lemma}

\begin{proof}
    We remark that the group algebra of any $p$-group over $k$ is a local $k$-algebra (for example, see \cite[Corollary 3.3]{MR860771}). For this reason, it is only trivial module $k_G$ that can be composition factor of ${}_{kG}k(G/N)$, where ${}_{kG}k(G/N)$ means the left $kG$-module with its basis $G/N$. Hence we get the following isomorphisms of $kG$-modules:
    \begin{align*}
        \induc_N^G \restr_N^G S
         & \cong \induc_N^G ((\restr_N^G S)\otimes k_N) \\
         & \cong S\otimes \induc_N^G k_N                \\
         & \cong S \otimes k(G/N).
    \end{align*}
    Therefore, all composition factors of the module in the right-hand side are isomorphic to $S\otimes k_G\cong S$.
\end{proof}

\begin{corollary}\label{corollary only one simple in induc}
    Let $T$ be a simple $kN$-module. Suppose that $G/N$ is a $p$-group, then the $kG$-module $\induc^G_N T$ has only one sort of simple module which can be a composition factor.
\end{corollary}

\begin{proof}
    We can take a simple $kG$-module $S$ satisfying $\Hom_{kG}(\induc_N^G T,S) \neq 0$. Since $\Hom_{kG}(\induc_N^G T,S)\cong \Hom_{kN}(T,\restr_N^G S)$ by Proposition \ref{Theorem: Frobenius and projective},
    the restriction module $\restr_N^G S$ has a submodule isomorphic to $T$.
    Hence the induced module $\induc_N^G T$ is isomorphic to a submodule of $\induc_N^G \restr_N^G S$. Therefore, by Lemma \ref{Lemma simple ind res}, the conclusion follows.
\end{proof}

\begin{lemma}\label{lemma indec proj}
    Let $N$ be a normal subgroup of $G$ such that $G/N$ is a $p$-group and $S$ a simple $kG$-module. Assume that $\restr_N^G S$ is a simple $kN$-module and denote this by $T$. Then the following hold:
    \begin{enumerate}
        \item $\decompgp_G(P(T))=G$,
        \item $\induc_N^G P(T) \cong P(S)$,
        \item $\restr_N^G P(S)\cong P(T)^{\oplus |G:N|}$.
    \end{enumerate}
\end{lemma}

\begin{proof}
    The assumption implies that $\decompgp_G(T)=G$ by Theorem \ref{Clifford's Theorem of simple module}. Hence, for any $g\in G$,
    we have $gT\cong T$, which implies that $gP(T)\cong P(gT)\cong P(T)$ and the first assertion is proved.
    Since the induced module $\induc_N^G P(T)$ is an indecomposable projective module by Proposition \ref{Theorem: Frobenius and projective} and Theorem \ref{Green's indecomposability theorem}, and
    \begin{equation*}
        \Hom_{kG}(\induc_N^G P(T),S)\cong \Hom_{kN}(P(T),\restr_N^G S)=\Hom_{kN}(P(S),S)\neq 0,
    \end{equation*}
    the second assertion is proved.
    The third assertion is trivial by previous two assertions and Remark \ref{Remark: normal mackey}.
\end{proof}

\subsection{Block theory}

We recall the definition of blocks of group algebras. Let $G$ be a finite group. The group algebra $kG$ has a unique decomposition
\begin{equation*}
    kG=B_1\times \cdots \times B_l,
\end{equation*}
into a direct product of subalgebras $B_i$ each of which is indecomposable as an algebra.
We call each indecomposable direct product component $B_i$ a block of $kG$ and the above decomposition the block decomposition. We remark that any block $B_i$ is an ideal of $kG$.

For any indecomposable $kG$-module $U$, there exists a unique block $B_i$ of $kG$ such that $U=B_iU$ and $B_jU=0$ for all $j\neq i$. Then we say that $U$ lies in the block $B_i$ or simply $U$ is a $B_i$-module. We denote by $B_0(kG)$ the principal block of $kG$ in which the trivial $kG$-module lies.

We recall the definition and basic properties of defect groups of blocks.
\begin{definition}
    Let $B$ be a block of $kG$.
    A {\it defect group} $D$ of $B$ is a minimal subgroup of $G$ satisfying the following condition:
    the $B$-bimodule epimorphism
    \begin{align*}
        \mu_D: B\otimes_{kD}B    & \rightarrow B             \\
        \beta _1\otimes \beta _2 & \mapsto \beta _1 \beta _2
    \end{align*}
    is a split epimorphism.
\end{definition}

\begin{proposition}[{\cite[Chapter 4, 5]{MR860771}}]
    Let $B$ be a block of $kG$ and $D$ a defect group of $B$. Then the following hold:
    \begin{itemize}
        \item $D$ is a $p$-subgroup of $G$ and the set of all defect groups of $B$ forms the conjugacy class of $D$ in $G$.
        \item $D$ is a cyclic group if and only if the algebra $B$ is finite representation type.
        \item If $B$ is the principal block of $kG$, then $D$ is a Sylow $p$-subgroup of $G$.
    \end{itemize}
\end{proposition}

\begin{theorem}[{\cite[Corollary 14.6, Theorem 17.1 and proof of Lemma 19.3]{MR860771}}]\label{Theorem:Dade Brauer Tree}
    Let $B$ be a block of $kG$ and $D$ a defect group of $B$.
    \begin{itemize}
        \item $D$ is trivial group if and only if $B$ is a simple algebra.
        \item $D$ is a non-trivial cyclic group if and only if $B$ is a Brauer tree algebra with $e$ edges and  multiplicity $(|D|-1)/e$ where $e$ is a devisor of $p-1$.
    \end{itemize}
\end{theorem}

Let $N$ be a normal subgroup of $G$, $b$ a block of $kN$ and $B$ a block of $kG$. We say that $B$ covers $b$ if $1_B 1_b\neq 0$. We denote by $\inertiagp_G(b)$ the {\it inertial group} of $b$ in $G$, that is $\inertiagp_G(b):=\left\{ g \in G \mid gbg^{-1} = b \right\}$.

\begin{remark}[{See \cite[section 15]{MR860771}}]\label{Remark:cover}
    With the above notation, the following are equivalent:
    \begin{enumerate}
        \item The block $B$ covers $b$.
        \item There exists a non-zero $B$-module $U$ such that $\restr_N^G U$ has a non-zero summand lying in $b$.
        \item For any non-zero $B$-module $U$, there exists a non-zero summand of $\restr_N^G U$ lying in $b$.
    \end{enumerate}
\end{remark}

\begin{remark}
    In particular, the principal block $B_0(kN)$ is covered by $B_0(kG)$ and $\inertiagp_G(B_0(kN))=G$.
\end{remark}

\begin{theorem}[{Clifford's Theorem for blocks \cite[Theorem 15.1, Lemma 15.3]{MR860771}}]\label{Clifford's Theorem block ver}
    Let $N$ be a normal subgroup of $G$, $b$ a block of $kN$, $B$ a block of $kG$ covering $b$ and $U$ a $B$-module. Then the following hold:
    \begin{enumerate}
        \item The set of blocks of $kN$ covered by $B$ equals to the conjugacy class of $b$ in $G$; that is, $$\left\{ b' \mid \text{$b'$ is a block of $kN$ covered by $B$} \right\}=\left\{ gbg^{-1} \mid g \in G \right\}.$$
        \item We get the following isomorphism of $kN$-modules:
              \begin{equation*}
                  \restr_N^G U \cong \bigoplus_{g \in [G/\inertiagp_G(b)]}gbU.
              \end{equation*}
        \item We can consider $bU$ as a $k\inertiagp_G(b)$-module and $U \cong \induc_{\decompgp_G(b)}^G bU$.
    \end{enumerate}
\end{theorem}

\begin{proposition}[{See \cite[Theorem 6.8.3]{MR3821517}}]\label{Morita equivalence covering block}
    Let $N$ be a normal subgroup of $G$, $b$ a block of $kN$ and $H$ a subgroup of $G$ satisfying the condition $N \leq H \leq \inertiagp_G(b)$. Let $\beta$ be a block of $kH$ and $B$ a block of $kG$ both covering $b$. If $B$ is a unique block of $kG$ covering $b$, then the following hold:
    \begin{enumerate}
        \item The induction functor $\induc_H^G \colon kH\lmod \rightarrow kG\lmod$ restricts to a faithful functor $$\induc_H^G \colon \beta\lmod \rightarrow B\lmod.$$\label{2020-03-25 15:17:09}
        \item In addition, if it holds that $H=\inertiagp_G(b)$, then the induction functor in \eqref{2020-03-25 15:17:09} is a Morita equivalence.
    \end{enumerate}
\end{proposition}

\begin{proof}
    (See \cite[Proof of Lemma 19.6]{MR860771}.)
    First we prove that the induced module $\induc_H^G V$ is a $B$-module for any $\beta$-module $V$.
    Since $\inertiagp_H(b)=\inertiagp_G(b)\cap H=H$, the restriction module $\restr_N^H V$ is a $b$-module  by Theorem \ref{Clifford's Theorem block ver}. Let $W$ be an indecomposable summand of $\induc_H^G V$ and $A_W$ a block of $kG$ which $W$ lies in. Since
    \begin{equation*}
        \restr_N^G W \mid \restr_N^G \induc_H^G V \cong \bigoplus _{g\in [G/H]} g\restr_N^H V,
    \end{equation*}
    and the restriction module $\restr_N^H V$ is a $b$-module, the block $A_W$ covers $b$. By the assumption of the uniqueness of the block of $kG$ covering $b$, we get $A_W=B$. Since the indecomposable summand $W$ of $\induc_H^G V$ is arbitrary, the first assertion is proven (since the induction functor $\induc_H^G$ is faithful). The second assertion follows directly from \cite[Theorem 5.5.12]{MR998775}.
\end{proof}

In the setting of Proposition \ref{Morita equivalence covering block}, we need the uniqueness of the block of $kG$ covering the block $b$ to consider the proposition. The following result assures that we can apply the proposition for the case where the quotient $G/N$ is a $p$-group.

\begin{proposition}[{See \cite[Corollary 5.5.6]{MR998775} or \cite[Proposition 6.8.11]{MR3821517}}]\label{p-power index covering uniqueness}
    Let $N$ be a normal subgroup of $G$ and $b$ a block of $kN$.
    If $G/N$ is a $p$-group, then there exists a unique block of $kG$ covering $b$.
\end{proposition}

\subsection{Lemmas}
In this section we give several lemmas which are used in proof of our main theorems.

\begin{lemma}\label{simple module lemma}
    Let $N$ be a normal subgroup of $G$, $b$ a block of $kN$ and $B$ a block of $kG$ covering $b$.
    If $G/N$ is a $p$-group and $\inertiagp_G(b)=G$, then the following conditions are equivalent.
    \begin{enumerate}
        \item For any simple $b$-module $T$, the inertial group $\decompgp_G(T)$ of $T$ in $G$ is equal to $G$.
        \item For any simple $B$-module $S$, the restriction module $\restr_N^G S$ is a simple $b$-module.
    \end{enumerate}
    In addition, if the conditions above hold, then the restriction functor $\restr_N^G$ induces a bijection between the set of isomorphism classes of simple $B$-modules and the one of simple $b$-modules.
\end{lemma}

\begin{proof}
    First, we prove that the first condition implies the second one. Let $S$ be a simple $B$-module. By Theorem \ref{Clifford's Theorem of simple module} and the assumption, there exists a simple $b$-module $T$ such that $\restr_N^G S \cong T^{\oplus r}$ for some $r\in \mathbb{N}$. Since $G/N$ is a $p$-group, by Proposition \ref{unique extension of simple module}, there exists a simple $kG$-module $S'$ such that $\restr_N^G S'$ is isomorphic to $T$. Since $G/N$ is a $p$-group again, by Lemma \ref{Lemma simple ind res}, the all composition factors of $\induc_N^G \restr_N^G S$ and $\induc_N^G \restr_N^G S'$ are isomorphic to one simple module. It implies that $S\cong S'$ by the Jordan--H\"older theorem and we conclude that the first assertion implies the second one.

    We next show that the second condition implies the first one. Let $T$ be a simple $b$-module. By Proposition \ref{Morita equivalence covering block}, \ref{p-power index covering uniqueness} the induced module $\induc_N^G T$ is a $B$-module and there exist a simple $B$-module $S$ such that $\Hom_B(\induc_N^G T, S)\neq 0$.
    By the assumption and Proposition \ref{Theorem: Frobenius and projective}, we have $T\cong \restr_N^G S$. Hence by Theorem \ref{Clifford's Theorem of simple module}, we have $\inertiagp_G(T)=G$.
    Therefore we have proven that the second assertion implies the first one.

    The remaining deduction is immediate from the fact that the above two conditions are equivalent and from Proposition \ref{unique extension of simple module}.
\end{proof}

\begin{lemma}\label{induction functor projective cover}
    Let $N$ be a normal subgroup of $G$ such that $G/N$ is a $p$-group.
    For any $kN$-module $V$, the following hold:
    \begin{enumerate}
        \item $P(\induc_N^G V)\cong \induc_N^G P(V)$,\label{2020-03-25 01:17:41}
        \item $\Omega(\induc_N^G V)\cong \induc_N^G \Omega(V)$,\label{2020-03-25 01:17:47}
        \item $\tau\induc_N^G V\cong \induc_N^G \tau V$.\label{2020-03-25 01:17:54}
    \end{enumerate}
\end{lemma}

\begin{proof}
    It is clear that if \eqref{2020-03-25 01:17:41}, \eqref{2020-03-25 01:17:47} and \eqref{2020-03-25 01:17:54} are true for every indecomposable $kN$-module, then also they are true for every $kN$-module. Hence we only have to consider the case where $V$ is indecomposable.
    There exists a projective $kG$-module $Q$ such that $P(\induc_N^G V)\oplus Q \cong \induc_N^G P(V)$ and that $\Omega (\induc_N^G V)\oplus Q \cong \induc_N^G \Omega(V)$ Since the $kN$-module $\Omega V$ is indecomposable, the $kG$-module $\induc_N^G \Omega V$ is also indecomposable by Theorem \ref{Green's indecomposability theorem}.
    This implies that $Q=0$ and hence we have \eqref{2020-03-25 01:17:41} and \eqref{2020-03-25 01:17:47}. Since $\tau V\cong \Omega \Omega V$, \eqref{2020-03-25 01:17:54} follows immediately from \eqref{2020-03-25 01:17:47}.
\end{proof}

\begin{lemma}[{\cite[Lemma 2.2]{MR2592757}}]\label{lemma number of simple}
    Let $N$ be a normal subgroup of $G$ and $b$ a block of $kN$.
    If $G/N$ is a $p$-group and the number of simple $b$-modules is strictly smaller than $p$, then for any simple $b$-module $S$, it holds that $\decompgp_G(S)=\inertiagp_G(b)$.
\end{lemma}

The next lemma given by Lemma \ref{lemma number of simple} helps us to prove our second main theorem from our first one.

\begin{lemma}\label{Cyclicdefect p-power index}
    Let $N$ be a normal subgroup of $G$  and $b$ a block of $kN$ with a cyclic defect group. If $G/N$ is a $p$-group, then we have $\decompgp_G(U)=\inertiagp_G(b)$ for any indecomposable $b$-module $U$.
\end{lemma}

\begin{proof}
    If the block $b$ is a simple algebra, {then the consequence is trivial because there is only one indecomposable $b$-module.}
    Therefore, we consider the case where $b$ is a Brauer tree algebra.
    Let $U$ be an indecomposable $b$-module.
    We prove that $\decompgp_G(U)=\inertiagp_G(b)$ by induction on the composition length of $U$.
    First, assume that the composition length of $U$ is one, that is, $U$ is a simple $b$-module. Then,
    by Theorem \ref{Theorem:Dade Brauer Tree}, we have that the number of the isomorphism classes of simple $b$-modules is strictly smaller than $p$, so we get that $\decompgp_G(U)=\inertiagp_G(b)$ from Lemma \ref{lemma number of simple}.

    Now suppose that the composition length of $U$ is two or more. We remark that any indecomposable $b$-module is a string module (for exaple, see \cite{MR3823391}).
    Hence, we can take a simple $b$-module $S$ and an indecomposable $b$-module $V$ which satisfy at least one of the following conditions:
    \begin{itemize}
        \item There exists an exact sequence
              \begin{equation*}
                  \xymatrix{
                  0 \ar[r] & S \ar[r]^-{\mu} & U \ar[r]^\nu & V  \ar[r] & 0.
                  }
              \end{equation*}
        \item There exists an exact sequence
              \begin{equation*}
                  \xymatrix{
                  0 \ar[r] & V \ar[r]^-{\mu'} & U \ar[r]^{\nu'} & S \ar[r] & 0.
                  }
              \end{equation*}
    \end{itemize}
    It suffices to prove $\decompgp_G(U)=\inertiagp_G(b)$ under the assumption that there exists the first exact sequence, the other case being proved similarly.
    For any $g\in G$, we take $kN$-module isomorphisms $\varphi \colon gS \rightarrow S$ and $\psi \colon V \rightarrow gV$ by the induction hypothesis.
    We obtain the following commutative diagram:
    \begin{equation*}
        \xymatrix{
        0 \ar[r] & gS \ar@{}[rd]|{P.O.}\ar[d]_-{\varphi} \ar[r]^-{g\mu g^{-1}} & gU \ar[d]_-{\varphi'} \ar[r]^-{g\nu g^{-1}} & gV \ar@{=}[d]_-{\mathrm{id}} \ar[r] & 0 \\
        0 \ar[r] & S \ar@{=}[d]_-{\mathrm{id}} \ar[r]^-{\varepsilon_1 } & X \ar@{}[rd]|{P.B.}  \ar[r]^-{\sigma_1} & gV \ar[r]  & 0 \\
        0 \ar[r] &S \ar@{}[rd]|{P.O.} \ar[d]_-{t} \ar[r]^-{\varepsilon_2} &Y \ar[u]^-{\psi'} \ar[d]_-{t'} \ar[r]^-{\sigma_2 } &V \ar[u]^-{\psi} \ar[r] \ar@{=}[d]_-{\mathrm{id}} &0\\
        0 \ar[r] &S\ar[r]^-{\varepsilon }&U\ar[r]^-{\sigma }&V \ar[r] &0
        }
    \end{equation*}
    where $t$ is a scalar map since $\dim_k\Ext_b^1(V,S)=1$ (see \cite[Proposition 21.7]{MR860771}). Therefore we get $gU\cong U$.
\end{proof}

The following lemmas and corollaries obtained by them are used in the proof of Theorem \ref{Main Theorem: poset struc}.

\begin{lemma}\label{Lemma; tau rigid iff}
    Let $N$ be a normal subgroup of $G$ such that $G/N$ is a $p$-group, $b$ a block of $kN$ satisfying $\inertiagp_G(b)=G$ and $B$ the unique block of $kG$ covering $b$.
    For a $\tau$-rigid $b$-module $U$, the induced module $\induc_N^G U$ is $\tau$-rigid if and only if $\Hom_b(gU,\tau U)=0$ for all $g \in G$. In particular, for a $G$-invariant $\tau$-rigid $b$-module $U$, the induced module $\induc_N^G U$ is also a $\tau$-rigid $B$-module.
\end{lemma}

\begin{proof}
    Let $U$ be a $\tau$-rigid $b$-module. Then the $kG$-module $\induc_N^G U$ is a $B$-module by Proposition \ref{Morita equivalence covering block} and Lemma \ref{p-power index covering uniqueness}. By Lemma \ref{induction functor projective cover}, Proposition \ref{Theorem: Frobenius and projective} and Theorem \ref{Mackey's decomposition formula}, we get the following isomorphisms:
    \begin{align*}
        \Hom_B(\induc_N^G U,\tau \induc_N^G U)
         & \cong \Hom_B(\induc_N^G U, \induc_N^G \tau U)    \\
         & \cong \Hom_b(\restr_N^G \induc_N^G U, \tau U)    \\
         & \cong \bigoplus_{g \in [G/N]}\Hom_b(gU, \tau U).
    \end{align*}
    It concludes the proof.
\end{proof}

\begin{lemma}\label{tau rigid pair hom zero}
    Let $N$ be a normal subgroup of $G$ such that $G/N$ is a $p$-group, $b$ a block of $kN$ satisfying $\inertiagp_G(b)=G$ and $B$ the unique block of $kG$ covering $b$.
    Let $U$ be a $b$-module and $P$ a projective $b$-module. If the pair $(U,P)$ satisfies $\Hom_b(P,U)=0$, then we have $\Hom_B(\induc_N^G P, \induc_N^G U)=0$.
\end{lemma}

\begin{proof}
    By Proposition \ref{Theorem: Frobenius and projective}, Theorem \ref{Mackey's decomposition formula} and Lemma \ref{lemma indec proj}, we get the following isomorphisms:
    \begin{align*}
        \Hom_B(\induc_N^G P, \induc_N^G U)
         & \cong \Hom_b(\restr_N^G \induc_N^G P, U)   \\
         & \cong \Hom_b(\bigoplus_{g \in[G/N]} gP, U) \\
         & \cong \bigoplus_{g \in[G/N]}\Hom_b(P, U)   \\
         & = 0.
    \end{align*}
\end{proof}

\begin{lemma}\label{induced iso}
    Let $N$ be a normal subgroup of $G$, and $U$, $U'$ indecomposable $kN$-modules.
    If the induced module $\induc_N^G U$ is isomorphic to $\induc_N^G U'$, then $U$ is isomorphic to $gU'$ for some $g \in G$. In particular, if $U$ is $G$-invariant and the induced module $\induc_N^G U$ is isomorphic to $\induc_N^G U'$, then $U$ is isomorphic to $U'$.
\end{lemma}

\begin{proof}
    Let $U$ and $U'$ be indecomposable $kN$-modules with $\induc_N^G U$ isomorphic to $\induc_N^G U'$. Then, by Theorem \ref{Mackey's decomposition formula}, we have
    \begin{equation*}
        U \mid \restr_N^G \induc_N^G U \cong \restr_N^G \induc_N^G U'\cong \bigoplus_{g \in [G/N]} gU'.
    \end{equation*}
    By the Krull--Schmidt Theorem, we get $U\cong gU'$ for some $g \in G$.
\end{proof}

\section{Proof of the main theorems}\label{section: proof of the main theorems}
In this section, we give proofs of our main theorems.
The next lemma has a key role.

\begin{lemma}\label{Lemma nice functor}
    Let $\Lambda$ and $\Gamma$ be finite dimensional $k$-algebras with the same numbers of isomorphism classes of the simple modules. Assume an exact functor $F$ from $\Lambda\lmod$ to $\Gamma\lmod$ satisfies the following conditions:
    \begin{enumerate}[label=(\roman*)
            ,font=\rm]
        \item The functor $F$ preserves indecomposability, projectivity and $\tau$-rigidity.
        \item If $\Hom_{\Lambda}(P,M)=0$ then $\Hom_{\Gamma}(F(P),F(M))=0$ for any projective $\Lambda$-module $P$ and $\Lambda$-module $M$.
        \item The functor $F$ induces an injection
              from the set of isomorphism classes of indecomposable modules over $\Lambda$ to the one over $\Gamma$.
    \end{enumerate}
    Then the following hold:
    \begin{enumerate}
        \item The functor $F$ induces an embedding of $\mathcal{H}(\stautilt \Lambda)$ into $\mathcal{H}(\stautilt \Gamma)$ which sends any connected component of $\mathcal{H}(\stautilt \Lambda)$ into $\mathcal{H}(\stautilt \Gamma)$ as a connected component.
        \item If $\Lambda$ is a support $\tau$-tilting finite algebra, then the functor $F$ induces an isomorphism from $\stautilt \Lambda$ to $\stautilt \Gamma$ of partially ordered sets.
    \end{enumerate}
\end{lemma}

\begin{proof}
    We can easily see that $(F(M),F(P))$ is a support $\tau$-tilting pair (or almost complete support $\tau$-tilting pair) over $\Gamma$ for any support $\tau$-tilting pair (or almost complete support $\tau$-tilting pair, respectively) $(M,P)$ over $\Lambda$.
    Hence, the functor $F$ sends any support $\tau$-tilting $\Lambda$-module to a support $\tau$-tilting $\Gamma$-module.
    Now assume that support $\tau$-tilting $\Lambda$-modules $M_1$ and $M_2$ are support $\tau$-tilting mutation of each other.
    Let $(L,Q)$ be a basic almost complete support $\tau$-tilting pair appearing as a direct summand of both $(M_1,P_1)$ and $(M_2,P_2)$. Then the pairs $(F(M_1),F(P_1))$ and $(F(M_2),F(P_2))$ are distinct by the third assumption on $F$ and have an almost complete support $\tau$-tilting pair $(F(L),F(Q))$ as a direct summand.
    Therefore $(F(M_1),F(P_1))$ and $(F(M_2),F(P_2))$ are support $\tau$-tilting mutation of each other.
    Now assume $M_2 > M_1$ holds. Then by the definition of the partial order on $\stautilt \Lambda$, there exist $r \in \mathbb{N}$ and an epimorphism $M_2^{\oplus r}\xtwoheadrightarrow{f} M_1.$
    Since $F$ is an exact functor, we get an epimorphism
    $F(M_2)^{\oplus r}\xtwoheadrightarrow{F(f)} F(M_1)$
    which implies $F(M_2) > F(M_1)$.
    Hence we have that the functor $F$ embeds $\mathcal{H}(\stautilt \Lambda)$ into $\mathcal{H}(\stautilt \Gamma)$.
    By Remark \ref{2020-03-20 15:46:59}, we have that $\mathcal{H}(\stautilt \Lambda)$ is $|\Lambda|$-regular, so any connected component $C$ in $\mathcal{H}(\stautilt \Lambda)$ is also $|\Lambda|$-regular.
    Hence the image of $C$ under the embedding above is some connected $|\Gamma|$-regular subquiver in $\mathcal{H}(\stautilt \Gamma)$ because $|\Lambda|=|\Gamma|$ and so  is some connected component in $\mathcal{H}(\stautilt \Gamma)$.
    If $\Lambda$ is a support $\tau$-tilting finite algebra, then the image of $\mathcal{H}(\stautilt \Lambda)$ under the embedding above is a finite connected component in $\mathcal{H}(\stautilt \Gamma)$, which coincides with $\mathcal{H}(\stautilt \Gamma)$ by Proposition \ref{connected component}.
\end{proof}

From now on, let $k$ be an algebraically closed field of characteristic $p$.
Now we give a theorem implying our first main theorem.

\begin{theorem}\label{Main Theorem: poset struc}
    Let ${\widetilde G}$ be a finite group satisfying the conditions $G\unlhd {\widetilde G}$ and ${\widetilde G}/G$ is a $p$-group. Let $B$ a block of $kG$ and ${\widetilde B}$ the unique block of $k{\widetilde G}$ covering $B$.
    Assume that any indecomposable $B$-module is $\inertiagp_{\widetilde G}(B)$-invariant. Then we have the following:
    \begin{enumerate}
        \item The induction functor $\induc_G^{\widetilde G}$ induces an embedding of $\mathcal{H}(\stautilt B)$ into $\mathcal{H}(\stautilt {\widetilde B})$ and any connected component of $\mathcal{H}(\stautilt B)$ is embedded as a connected component of $\mathcal{H}(\stautilt {\widetilde B})$.
        \item If $B$ is a support $\tau$-tilting finite block, then the induction functor $\induc_G^{\widetilde G}$ induces an isomorphism from $\stautilt B$ to $\stautilt {\widetilde B}$ of partially ordered sets.
    \end{enumerate}
\end{theorem}

\begin{proof}
    Let $\beta$ be the block of $k\inertiagp_{\widetilde G}(B)$ covering $B$. The functors $\induc_G^{\inertiagp_{\widetilde G}(B)}\colon B\lmod\rightarrow \beta\lmod$ and $\induc_{\inertiagp_{\widetilde G}(B)}^{\widetilde G}\colon \beta \lmod \rightarrow {\widetilde B}\lmod$ 
    are exact and the latter induces is a Morita equivalence by Proposition \ref{Morita equivalence covering block}.
    We remark that the number of isomorphism classes of the simple $B$-modules is equal to the one of the simple $\beta$-modules by the assumption that any indecomposable $B$-module is $\inertiagp_{\widetilde G}(B)$-invariant and Lemma \ref{simple module lemma}.
    In order to accomplish the proof, it is enough to show that $\induc_G^{\inertiagp_{\widetilde G}(B)}$ satisfies the three conditions in Lemma \ref{Lemma nice functor}.
    The functor $\induc_G^{\inertiagp_{\widetilde G}(B)}$ preserves indecomposability, projectivity and $\tau$-rigidity by Theorem \ref{Green's indecomposability theorem}, Proposition \ref{Theorem: Frobenius and projective} and Lemma \ref{Lemma; tau rigid iff} respectively.
    By Lemma \ref{tau rigid pair hom zero} and Lemma \ref{induced iso}, the functor $\induc_G^{\inertiagp_{\widetilde G}(B)}$ satisfies the second and third condition in Lemma \ref{Lemma nice functor}. Therefore we have completed the proof.
\end{proof}

\begin{proof}[Proof of Theorem \ref{Main theorem brauer tree}]
    It is immediate from Theorem \ref{Main Theorem: poset struc} and Lemma \ref{Cyclicdefect p-power index}.
\end{proof}

\begin{corollary}\label{Induction 2tilt}
    Let $G$, ${\widetilde G}$, $B$ and ${\widetilde B}$ be the same as in Theorem \ref{Main Theorem: poset struc}.
    With the same assumption in Theorem \ref{Main Theorem: poset struc}, the induction functor $\induc_G^{\widetilde G}$ induces a partially ordered set morphism from $\twotilt B$ to $\twotilt {\widetilde B}$ which makes the following diagram
    \begin{equation*}
        \xymatrix{
        \stautilt B \ar[r]^-{\induc_G^{\widetilde{G}}} \ar[d]_-[@]{\sim} \ar@{}[dr]|\circlearrowleft &\stautilt {\widetilde B} \ar[d]^-[@]{\sim}\\
        \twotilt B \ar[r]_-{\induc_G^{\widetilde{G}}}& \twotilt {\widetilde B}
        }
    \end{equation*}
    of partially ordered sets commutative where the vertical isomorphisms are given by \cite{MR3187626} and the upper horizontal morphism is given by Theorem \ref{Main Theorem: poset struc}.
\end{corollary}
\begin{proof}
    Let $(M, P) \in \stautilt B$ and $P_1 \xrightarrow{f_1} P_0 \xrightarrow{f_0} M\rightarrow 0$ the minimal projective presentation of $M$.
    By Theorem \ref{2020-03-27 17:22:17}, the corresponding two-term tilting complex is $P_1\oplus P \xrightarrow{(f_1\ 0)} P_0$.
    Moreover, $(\induc_G^{\widetilde G} M, \induc_G^{\widetilde G} P)$ is a support $\tau$-tilting pair over $\widetilde{B}$ by Theorem \ref{Main Theorem: poset struc} and the sequence $\induc_G^{\widetilde G} P_1 \xrightarrow{\induc_G^{\widetilde G} f_1} \induc_G^{\widetilde G} P_0 \xrightarrow{\induc_G^{\widetilde G} f_0} \induc_G^{\widetilde G} M\rightarrow 0$ is also the minimal projective presentation of $\induc_G^{\widetilde G} M$ by Proposition \ref{induction functor projective cover}. Therefore, the corresponding two-term tilting complex is $\induc_G^{\widetilde G} P_1\oplus \induc_G^{\widetilde G} P \xrightarrow{(\induc_G^{\widetilde G} f_1\ 0)} \induc_G^{\widetilde G} P_0$. It concludes the proof.
\end{proof}

\section{Applications and Examples}\label{section:Applications}

\begin{example}
    Let $p$ be an odd prime, $n$ a positive integer and $G=D_{p^n}\cong C_{p^n}\rtimes C_2$ the dihedral group with order $2p^n$.
    Then the automorphism group $\mathrm{Aut}(G)$ of $D_{p^n}$ is isomorphic to $C_{p^n}\rtimes (C_{p^{n-1}}\times C_{p-1})$.
    Let $Q$ be a nontrivial $p$-subgroup of $\mathrm{Aut}(G)$.
    We denote the group $G\rtimes Q$ by ${\widetilde G}$.
    The group algebra $kG$ is indecomposable as an algebra, so it is the unique block of $kG$ with cyclic defect group.
    There are two simple $kG$-module isomorphism classes $k_G$ and $s_G$ where $k_G$ is the trivial module and $s_G$ is the sign module, and the block $kG$ is a basic Brauer tree algebra associated to the following Brauer tree.
    \vspace{10pt}

    \begin{tikzpicture}[auto,
            node_style/.style={circle,draw=black,thick},
            exp_node_style/.style={circle,draw=black,fill=black,thick},
            edge_style/.style={draw=black, ultra thick}]
        \node[node_style] (e1) at (1,1.5) {};
        \node[exp_node_style] (e2) at (4,1.5) {};
        \node[node_style] (e3) at (7,1.5) {};
        \node(multi) at (10,1.5){: multiplicity $(p^n-1)/2$};

        \draw[edge_style]  (e1) edge node{$k_G$} (e2);
        \draw[edge_style]  (e2) edge node{$s_G$} (e3);
    \end{tikzpicture}
    \vspace{20pt}

    We can calculate $\mathcal{H}(\stautilt kG)$ (see \cite{MR3461065}, \cite{aoki2018torsion}).

    \vspace{10pt}
    \begin{center}
        \begin{tikzpicture}[auto]
            \node(projene) at (3,3.6){$kG$};
            \node(PkGkG) at (1,2.4){$P(k_G)\oplus k_G$};
            \node(PsGsG) at (5,2.4){$P(s_G)\oplus s_G$};
            \node(kG) at (1,1){$k_G$};
            \node(sG) at (5,1){$s_G$};
            \node(zero) at (3,0){$0$};
            \node(multi) at (1,4){$\mathcal{H}(\stautilt kG)$:};
            \draw[->](projene) edge node{} (PkGkG);
            \draw[->](projene) edge node{} (PsGsG);
            \draw[->](PkGkG) edge node{} (kG);
            \draw[->](PsGsG) edge node{} (sG);
            \draw[->](kG) edge node{} (zero);
            \draw[->](sG) edge node{} (zero);
        \end{tikzpicture}
    \end{center}

    The group algebra $k{\widetilde G}$ is also indecomposable as an algebra, hence it is the block of $k{\widetilde G}$ covering $kG$.
    Since $\stautilt kG \cong \stautilt k{\widetilde G}$ by Theorem \ref{Main theorem brauer tree}, although $k{\widetilde G}$ is of wild representation type, we can give $\mathcal{H}(\stautilt k{\widetilde G})$ explicitly.

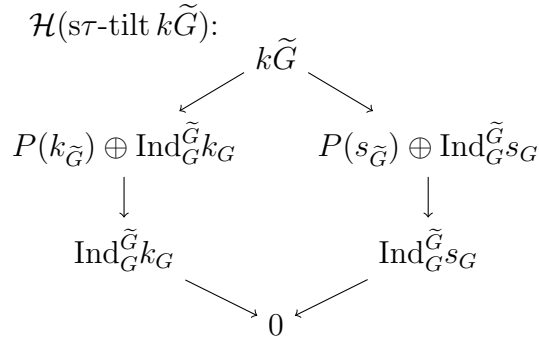
    \begin{center}
        \begin{tikzpicture}[auto]
            \node(projene) at (3,3.6){$k{\widetilde G}$};
            \node(PkGkG) at (1,2.4){$P(k_{\widetilde G})\oplus \induc_G^{\widetilde G}k_G$};
            \node(PsGsG) at (5,2.4){$P(s_{\widetilde G})\oplus \induc_G^{\widetilde G}s_G$};
            \node(kG) at (1,1){$\induc_G^{\widetilde G}k_G$};
            \node(sG) at (5,1){$\induc_G^{\widetilde G}s_G$};
            \node(zero) at (3,0){$0$};
            \node(multi) at (1,4){$\mathcal{H}(\stautilt k{\widetilde G})$:};

            \draw[->](projene) edge node{} (PkGkG);
            \draw[->](projene) edge node{} (PsGsG);
            \draw[->](PkGkG) edge node{} (kG);
            \draw[->](PsGsG) edge node{} (sG);
            \draw[->](kG) edge node{} (zero);
            \draw[->](sG) edge node{} (zero);
        \end{tikzpicture}
    \end{center}

    According to \cite{kase2017support}, since  $\stautilt kG$ is isomorphic to $\stautilt k{\widetilde G}$ as poset, the Gabriel quiver of $k{\widetilde G}$ coincides with the one of $kG$ except for their all loops.
    In fact, the Gabriel quiver of $kG$ is as follows.
    
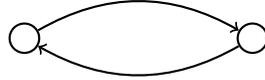
\begin{center}
        \begin{tikzpicture}[auto,
                node_style/.style={circle,draw=black,thick},
                edge_style/.style={->,bend left,draw=black, thick}]
            \node[node_style] (e1) at (1,1.5) {};
            \node[node_style] (e2) at (4,1.5) {};
            \draw[edge_style]  (e1) edge node{} (e2);
            \draw[edge_style]  (e2) edge node{} (e1);
        \end{tikzpicture}
    \end{center}

    If the nontrivial $p$-subgroup $Q$ of $\mathrm{Aut}(G)$ is cyclic, then the Gabriel quiver of $k{\widetilde G}$ is as follows.
    
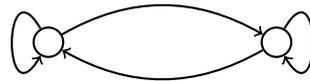
\begin{center}
        \begin{tikzpicture}[auto,
                node_style/.style={circle,draw=black,thick},
                edge_style/.style={->,bend left,draw=black, thick}]
            \node[node_style] (e1) at (1,1.5) {};
            \node[node_style] (e2) at (4,1.5) {};
            \draw[edge_style]  (e1) edge node{} (e2);
            \draw[edge_style]  (e2) edge node{} (e1);
            \path (e1)  edge[->,out=120,in=240,min distance=10mm,draw=black, thick] (e1);
            \path (e2)  edge[->,out=60,in=300,min distance=10mm,draw=black, thick] (e2);
        \end{tikzpicture}
    \end{center}

    On the other hand, if the nontrivial $p$-subgroup $Q$ of $\mathrm{Aut}(G)$ is non-cyclic, then the Gabriel quiver of $k{\widetilde G}$ is as follows.
    
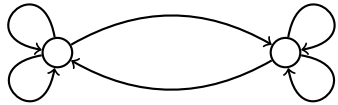
\begin{center}
        \begin{tikzpicture}[auto,
                node_style/.style={circle,draw=black,thick},
                edge_style/.style={->,bend left,draw=black, thick}]
            \node[node_style] (e1) at (1,1.5) {};
            \node[node_style] (e2) at (4,1.5) {};
            \draw[edge_style]  (e1) edge node{} (e2);
            \draw[edge_style]  (e2) edge node{} (e1);
            \path (e1)  edge[->,out=100,in=170,min distance=10mm,draw=black, thick] (e1);
            \path (e1)  edge[->,out=190,in=260,min distance=10mm,draw=black, thick] (e1);
            \path (e2)  edge[->,out=80,in=10,min distance=10mm,draw=black, thick] (e2);
            \path (e2)  edge[->,out=350,in=280,min distance=10mm,draw=black, thick] (e2);
        \end{tikzpicture}
    \end{center}

\end{example}

\section*{Acknowledgements}
The authors would like to thank Naoko Kunugi for giving valuable comments.
The second author is grateful to Ryoichi Kase and Yuya Mizuno for their kind advice and useful discussions.


\providecommand{\bysame}{\leavevmode\hbox to3em{\hrulefill}\thinspace}
\providecommand{\MR}{\relax\ifhmode\unskip\space\fi MR }
\providecommand{\MRhref}[2]{%
    \href{http://www.ams.org/mathscinet-getitem?mr=#1}{#2}
}
\providecommand{\href}[2]{#2}

\end{document}